\documentclass[a4paper,oneside,12pt,reqno]{amsart}
\usepackage{hyperref}
\usepackage[headinclude,DIV13]{typearea}
\areaset{15.1cm}{25.0cm}
\usepackage{amsfonts,amssymb,amsmath,amsthm,bbm}
\usepackage[utf8]{inputenc}
\usepackage{graphicx, psfrag}

\usepackage{amsfonts,amssymb,mathrsfs}
\usepackage{stmaryrd}

\newtheorem{theorem}{Theorem}[section]
\newtheorem{lemma}[theorem]{Lemma}
\newtheorem{proposition}[theorem]{Proposition}
\newtheorem{corollary}[theorem]{Corollary}

\theoremstyle{definition}

\theoremstyle{remark}
\newtheorem*{remark}{Remark}

\def\paragraph#1{\noindent \textbf{#1}}

\providecommand{\abs}[1]{\lvert#1\rvert}
\providecommand{\norm}[1]{\lVert#1\rVert}
\providecommand{\Bigabs}[1]{\Big\lvert#1\Big\rvert}

\numberwithin{equation}{section}

\def\Id{\mathrm{Id}}

\def\d{\mathrm{d}}

\def\<{\langle}
\def\>{\rangle}
\def\a{\alpha}

\def\e{\epsilon}

\def\g{\gamma}

\def\l{\lambda}
\def\r{\rho}

\def\t{\tau}

\def\o{\omega}
\def\D{\Delta}
\def\L{\Lambda}
\def\G{\Gamma}

\def\S{\Sigma}

\def\del{\partial}
\def\R{{\Bbb R}}  
\def\P{{\Bbb P}}  
\def\Z{{\Bbb Z}}  
\def\C{{\Bbb C}}  
\def\E{{\Bbb E}}  

\let\cal=\mathcal

\def\FF{{\cal F}}

\def\HH{{\cal H}}

\def\OO{{\cal O}}

\def \G {{\Gamma}}
\def \L {{\Lambda}}

\def \e {{\varepsilon}}

\def \m {{\mu}}
\def \n {{\nu}}
\def \D {{\Delta}}
\def \t {{\tau}}

\def \g {{\gamma}}
\def \l {{\lambda}}
\def \d {{\delta}}
\def \a {{\alpha}}
\def \o {{\omega}}

\def \n {{\nu}}

\def \del {{\partial}}

\def \r {{\rho}}

\def \ba {\begin{array}}
\def \ea {\end{array}}

\def \D {{\mathbb D}}

\newcommand{\be}{\begin{equation}}
\newcommand{\ee}{\end{equation}}

\newcommand{\bea}{\begin{eqnarray}}
\newcommand{\eea}{\end{eqnarray}}
\def\TH(#1){\label{#1}}\def\thv(#1){\ref{#1}}
\def\Eq(#1){\label{#1}}\def\eqv(#1){(\ref{#1})}

\def\capa{\hbox{\rm cap}}

\def\wt {\widetilde}
\def\wh{\widehat}

\numberwithin{equation}{section}

\begin{document}

\title{Uniform estimates for metastable transition times in a
coupled bistable system}
\author{Florent Barret}
\address{CMAP UMR 7641, \'Ecole Polytechnique CNRS, Route de Saclay,
91128 Palaiseau Cedex France; {\rm email:
barret@cmap.polytechnique.fr}}
\author{ Anton Bovier}
\address{Institut f\"ur Angewandte Mathematik,
Rheinische Friedrich-Wilhelms-Universit\"at, Endenicher Allee 60,
53115 Bonn, Germany; {\rm email: bovier@uni-bonn.de}}
\author{Sylvie M\'el\'eard}
\address{CMAP UMR 7641, \'Ecole Polytechnique CNRS, Route de Saclay,
91128 Palaiseau Cedex France; {\rm email:
sylvie.meleard@polytechnique.edu}}

\maketitle

\begin{abstract}
We consider a  coupled bistable $N$-particle system  on $\R^N$
driven by a Brownian noise, with a strong coupling corresponding
to the synchronised
regime. Our aim  is to obtain sharp estimates on the
metastable transition times between the two stable states, both for
fixed $N$ and in the limit when $N$ tends to infinity, with error estimates uniform in $N$. These
estimates are a main  step towards a rigorous understanding of
the metastable behavior of infinite dimensional systems, such as the
stochastically perturbed Ginzburg-Landau equation.
Our results are based on the potential theoretic approach to metastability.
\end{abstract}

\bigskip


\section{Introduction}\label{section.intro}

The aim of this paper is to analyze the behavior
of metastable transition times for a gradient diffusion
model, independently of the dimension. Our method is based on potential theory and requires the existence of a reversible invariant probability measure. This measure exists for Brownian driven diffusions with gradient drift. 

To be specific, we consider  here a model of a chain of coupled
particles in a  double well potential driven by Brownian
noise (see e.g.  \cite{berglund107}). I.e., we consider the system of
stochastic differential equations 
\be\label{sde.1}
\mathrm{d}X_\epsilon(t)=-\nabla F_{\g,N}(X_\epsilon(t))\mathrm{d}t+
\sqrt{2\epsilon}\mathrm{d}B(t),
\ee 
where $X_\e(t)\in \R^N$ and
\be\label{potential.1} F_{\g,N}(x)=
\sum_{i\in\Lambda}\left(\frac14x_i^4-\frac12x_i^2\right)
+\frac{\gamma}{4}\sum_{i\in\Lambda}(x_i-x_{i+1})^2, 
\ee 
with
$\L=\Z/N\Z$ and  $\g>0$ is a parameter. $B$ is a $N$ dimensional
Brownian motion and $\e>0$ is the intensity of the noise. Each
component (particle) of this system is subject to force derived from a 
bistable potential. The components of the  system are coupled to their nearest neighbor with
intensity $\g$ and perturbed by independent noises of constant variance
$\e$. While the system without noise, i.e. $\e=0$, has
several stable fixpoints, for $\e>0$ transitions between these
fixpoints will occur at  suitable timescales. Such a situation is
called metastability.

For fixed $N$ and small $\e$, this problem has been widely studied
in the literature and we refer to the books by  Freidlin and
Wentzell \cite{freidlinwentzell} and Olivieri and Vares
\cite{OlivieriVares} for further discussions. In recent years, the
potential theoretic approach, initiated by Bovier, Eckhoff, Gayrard,
and Klein \cite{bovier04} (see \cite{bovier09} for a review), has
allowed to give very precise results on such transition times and
notably led to a proof of  the so-called Eyring-Kramers formula
which provides sharp asymptotics for these transition times, for any
fixed dimension. However, the results  obtained in
\cite{bovier04} do not include control of the error terms that are
uniform in the dimension of the system.

Our aim in this paper is to obtain such uniform estimates. These
estimates constitute  a the main  step towards a rigorous understanding of
the metastable behavior of infinite dimensional systems, i.e.
stochastic partial differential equations (SPDE) such as the
stochastically perturbed Ginzburg-Landau equation. Indeed, the
deterministic part  of the system \eqref{sde.1} can be seen as the
discretization of the drift part of this SPDE, as  has been
noticed e.g. in \cite{berglund207}. For a heuristic discussion of the
metastable behavior of this SPDE, see e.g. \cite{maierstein} and
\cite{west}. Rigorous results on the level of the  large deviation asymptotics
were obtained e.g. by Faris and Jona-Lasinio \cite{fajona},
Martinelli et al. \cite{martin}, and Brassesco \cite{stella}.

In the present paper we consider only the simplest situation, the
so-called synchronization
regime, where  the coupling $\g$ between the particles
is so strong  that there are only three relevant critical points of the
potential $F_{\g,N}$ \eqref{potential.1}. A generalization to  more complex
situations is however possible and
will be treated elsewhere.

The remainder of this paper is organized as follows. In Section 2 we recall
briefly the main results from the potential theoretic approach, we recall the
key properties of the potential $F_{\g,N}$, and we state the results on
metastability that follow from  the results of \cite{bovier04} for
fixed $N$. In Section 3 we deal with the case when $N$ tends to
infinity and state our main result, Theorem \ref{main}. In Section
4 we prove the main theorem through sharp estimates on the relevant
capacities.

\bigskip

In the remainder of the paper we adopt the following notations:
\begin{itemize}
\item for $t\in\R$, $\lfloor t \rfloor$ denotes the unique integer
$k$ such that $k\leq t<k+1$; 
\item $\t_D\equiv\inf \{t>0:X_t\in
D\}$ is the hitting time of the set $D$ for the process $(X_t)$;
\item $B_r(x)$ is the ball of radius $r>0$ and center $x\in\R^N$;
\item for
$p\geq 1$, and $(x_k)^N_{k=1}$ a sequence, we denote the
$L^P$-norm of $x$ by 
\be
\|x\|_p=\left(\sum^N_{k=1}|x_k|^p\right)^{1/p}. 
\ee
\end{itemize}

\bigskip

\noindent\textbf{Acknowledgments.} This paper is based on
the master thesis of F.B.\cite{barret07}
that was  written in part during a research visit
of F.B. at the International Research Training Group
``Stochastic models of Complex Systems'' at the Berlin University of Technology
under the supervision of A.B. F.B. thanks the IRTG SMCP
and TU Berlin for the kind hospitality and the ENS Cachan
for financial support. A.B.'s research is
supported in part by the German Research Council through
the SFB 611 and the Hausdorff Center for Mathematics.

\section{Preliminaries}\label{section.sharp}

\subsection{Key formulas from the potential theory approach}

We recall briefly the basic formulas from   potential theory
that we will need here.  The diffusion $X_{\epsilon}$ is the one introduced in \eqref{sde.1} and its infinitesimal generator is denoted by $L$. Note that $L$ is the closure of the operator
\be\Eq(generator)
L= \e e^{F_{\g,N}/\e} \nabla e^{-F_{\g,N}/\e}\nabla.
\ee
 For  $A,D$ regular
open
subsets of $\R^N$, let $h_{A,D}(x)$
be the harmonic  function (with respect to the generator $L$) with boundary conditions $1$ in $A$ and $0$ in $D$.
Then, for $x\in (A\cup D)^c$, one has
$h_{A,D}(x)=\P_x[\t_A<\t_D]$.  The equilibrium measure, $\,e_{A,D}$, is 
then defined (see e.g.
\cite{chungwalsh}) as the
unique measure on $\partial A$ such that \be \Eq(aa.1) h_{A,D}(x)=
\int_{\partial A}e^{-F_{\g,N}(y)/\e} G_{D^c}(x,y) e_{A,D}(dy), \ee
where $G_{D^c}$ is the Green function associated with the
generator $L$ on the domain $D^c$. This yields readily the following formula  
for the
hitting time of $D$ (see. e.g.  \cite{bovier04}):
\be\label{key.1} 
\int_{\partial
A}\E_{z}[\t_D]e^{-F_{\g,N}(z)/\e}e_{A,D}(dz) =
\int_{D^c}h_{A,D}(y)e^{-F_{\g,N}(y)/\e}dy. 
\ee
The capacity, $\capa(A,D)$, is defined as
\be\Eq(aa.2)
\capa(A,D)=\int_{\partial A} e^{-F_{\g,N}(z)/\e}e_{A,D}(dz).
\ee
Therefore,
\be\label{key.2}
\n_{A,D}(dz)=\frac{e^{-F_{\g,N}(z)/\e}e_{A,D}(dz)}{\capa(A,D)}
\ee
is a probability measure on $\del A$, that we may call the
equilibrium probability. The equation \eqref{key.1} then reads
\be\label{key.3}
\int_{\partial A}\E_z[\t_D]\n_{A,D}(dz) =
\E_{\n_{A,D}}[\t_D] =
\frac{\int_{D^c}h_{A,D}(y)e^{-F_{\g,N}(y)/\e}dy}{\capa(A,D)}.
\ee
The
strength of this formula comes from the fact that the capacity has an
alternative representation through the Dirichlet variational
principle (see e.g. \cite{fukushima}),
\be\label{cap.1}
\capa(A,D)=\inf_{h\in\HH}\Phi(h),
\ee where
\be\label{cap.2}
\HH=\Big\{h\in W^{1,2}(\mathbb{R}^N,
e^{-F_{\g,N}(u)/\e}du)\,|\,\forall z\,, h(z)\in[0,1]\,,
h_{|A}=1\,, h_{|D}=0\Big\},
\ee
and the Dirichlet form $\Phi$ is
given, for $h\in\HH$, as
\be\label{cap.3}
\Phi(h)=\e\int_{(A\cup
D)^c}e^{-F_{\g,N}(u)/\e}\norm{\nabla h(u)}_2^2du.
\ee

\begin{remark}
Formula \eqref{key.3} gives an average of the
mean transition time  with respect to the equilibrium
measure, that we will extensively use in what follows. A way to
obtain the quantity $\E_{z}[\t_D]$ consists in using H\"older and
Harnack estimates \cite{gt} (as developed in Corollary \ref{sharp-point})\cite{bovier04},
but it is far from obvious whether this can be extended to give estimates that are uniform in $N$.
\end{remark}

Formula \eqref{key.3} highlights the two terms for which we will
prove uniform estimates: the capacity (Proposition \ref{capacity}) and
the mass of $h_{A,D}$ (Proposition \ref{numerator}).

\subsection{Description of the Potential}

Let us describe in detail the potential $F_{\g,N}$, its
stationary points, and in particular the minima and  the 1-saddle
points, through which the transitions occur.

The coupling strength $\g$ specifies the geometry of $F_{\g,N}$. For
instance, if we set $\g=0$, we get a set of $N$ bistable
independent particles, thus the stationary points are
\be\label{potential.2}
x^*=(\xi_1,\dots,\xi_N)\quad\forall
i\in\llbracket1,N\rrbracket ,\, \xi_i\in\{-1,0,1\}.
\ee
To characterize their stability, we have to look to their Hessian
matrix whose signs of the eigenvalues give us the index saddle of
the point. It can be easily shown that, for $\gamma=0$, the minima
are those of the form \eqref{potential.2} with no zero coordinates
and the 1-saddle points have just one zero coordinate. As
$\gamma$ increases, the structure of the potential evolves and the
number of stationary points decreases from $3^N$ to $3$. We notice
that, for all $\gamma$, the points 
\be\label{potential.3}
I_{\pm}=\pm(1,1,\cdots,1)\quad O=(0,0,\cdots,0)
\ee 
are
stationary, fur\-ther\-more $I_{\pm}$ are minima. If we
calculate the Hessian at the point $O$, we have
\be\label{potential.4} 
\nabla^2F_{\gamma,N}(O)=
\begin{pmatrix}
-1+\gamma&-\frac{\gamma}{2}&0&\cdots&0&-\frac{\gamma}{2}\\
-\frac{\gamma}{2}&-1+\gamma&-\frac{\gamma}{2}&&&0\\
0&-\frac{\gamma}{2}&\ddots&\ddots&&\vdots\\
\vdots&&\ddots&\ddots&\ddots&0\\
0&&&\ddots&\ddots&-\frac{\gamma}{2}\\
-\frac{\gamma}{2}&0&\cdots&0&-\frac{\gamma}{2}&-1+\gamma
\end{pmatrix},
\ee
whose eigenvalues are, for all $\g>0$ and for $0\leq k\leq N-1$,
\be\label{eigenvalue.1}
\lambda_{k,N}=-\left(1-2\gamma\sin^2\left(\frac{k\pi}{N}\right)\right).
\ee
Set, for $k\geq1$, $\g_k^N=\frac{1}{2\sin^2\left(k\pi/N\right)}$.
Then these eigenvalues can be written in  the form
\be\label{eigenvalue.2}
\begin{cases}
\l_{k,N}&=\l_{N-k,N}=-1+\frac{\g}{\g^N_k},\, 1\leq k\leq N-1\\
\l_{0,N}&=\l_0=-1.
\end{cases}
\ee
Note that $(\g^N_k)_{k=1}^{\lfloor N/2\rfloor}$ is a
decreasing sequence, and so as $\g$ increases, the number of
non-positive eigenvalues $(\l_{k,N})_{k=0}^{N-1}$ decreases. When $\g>\g_1^N$,
the only negative eigenvalue is $-1$. Thus
\be
\g_1^N=\frac{1}{2\sin^2(\pi/N)}
\ee
is the threshold of the synchronization regime.

\begin{lemma}[Synchronization Regime]
If $\gamma>\g^N_1$, the only stationary points of $F_{\g,N}$
are $I_{\pm}$ and $O$. $I_\pm$ are minima, $O$ is a 1-saddle.
\end{lemma}

This lemma was proven in \cite{berglund107} by using a Lyapunov function.
This configuration is called
the synchronization regime because the coupling between the
particles is so strong that they all pass simultaneously through their
respective saddle points in a transition between the stable
equilibria ($I_\pm$).

In this paper, we will focus on this regime.

\subsection{Results for fixed $N$}

Let $\r>0$ and set $B_\pm\equiv B_{\r}(I_\pm)$, where $B_\r(x)$ denotes the 
ball of radius $\r$ centered at $x$. 
Equation \eqref{key.3} gives, with $A=B_-$ and $D=B_+$,
\be\label{key.4}
\E_{\n_{B_-,B_+}}[\t_{B_+}] =
\frac{\int_{B_+^c}h_{B_-,B_+}(y)e^{-F_{\g,N}(y)/\e}dy}{\capa(B_-,B_+)}.
\ee
First, we obtain a sharp estimate for this transition time
for fixed $N$:
\begin{theorem}\label{sharp}
Let $N>2$ be given. For $\g>\g^N_1=\frac{1}{2\sin^2(\pi/N)}$,
let $\sqrt{N}>\r\geq \epsilon>0$. Then
\be\label{sharp.1}
\mathbb{E}_{\n_{B_-,B_+}}[\tau_{B_+}]=2\pi c_Ne^{\frac{N}{4\epsilon}}(1+O(\sqrt{\e|\ln\e|^3}))
\ee
with
\be\label{cn.1}
c_N=\bigg[1-\frac{3}{2+2\g}\bigg]^{\frac{e(N)}{2}}
\prod_{k=1}^{\lfloor\frac{N-1}{2}\rfloor}
\bigg[1-\frac{3}{2+\frac{\g}{\g^N_k}}\bigg]
\ee
where $e(N)=1$ if $N$ is even and $0$ if $N$ is odd.
\end{theorem}

\begin{remark} The power $3$ at $\ln \e$ is missing in \cite{bovier04} 
by mistake.
\end{remark}

\begin{remark}
As mentioned above, for any fixed dimension, we can replace the
probability measure $\nu_{B_-,B_+}$ by the Dirac measure on the 
single point $I_-$,
using H\"older and Harnack inequalities \cite{bovier04}. 
This gives the  following
corollary:
\begin{corollary}\label{sharp-point}
Under the assumptions of Theorem \ref{sharp}, there exists $\a>0$
such that
\be\label{sharp-point.1}
\mathbb{E}_{I_-}[\tau_{B_+}]=2\pi
c_Ne^{\frac{N}{4\epsilon}}(1+O(\sqrt{\e|\ln\e|^3}). 
\ee
\end{corollary}
\end{remark}

\begin{proof} [Proof of the theorem]
We apply Theorem 3.2 in \cite{bovier04}. For
$\g>\g^N_1=\frac{1}{2\sin^2(\pi/N)}$, let us recall that  there
are only three stationary points: two minima $I_{\pm}$ and one
saddle point $O$. One easily checks that $F_{\g,N}$ satisfies
the following assumptions:
\begin{itemize}
\item $F_{\g,N}$ is polynomial in the $(x_i)_{i\in\Lambda}$ and
so clearly $C^3$ on $\mathbb{R}^N$. \item
$F_{\g,N}(x)\geq\frac{1}{4}\sum_{i\in\Lambda}x_i^4$ so
$F_{\g,N}\underset{x\rightarrow\infty}{\longrightarrow}+\infty$.
\item $\norm{\nabla F_{\g,N}(x)}_2\sim\norm{x}_3^3$ as
$\norm{x}_2\to\infty$. \item As $\Delta
F_{\g,N}(x)\sim3\norm{x}_2^2$ ($\norm{x}_2\to\infty$), then
$\norm{\nabla F_{\g,N}(x)}-2\Delta
F_{\g,N}(x)\sim\norm{x}_3^3$.
\end{itemize}

The Hessian matrix at the minima $I_\pm$ has the form
\be\label{potential.5}
\nabla^2F_{\g,N}(I_{\pm})= \nabla^2F_{\g,N}(O)+3 \Id,
\ee
whose eigenvalues are simply
\be\label{eigenvalue.3}
\n_{k,N}=\l_{k,N}+3.
\ee
%
Then Theorem 3.1 of \cite{bovier04} can be applied and yields, for
$\sqrt{N}>\r>\epsilon>0$, (recall the the negative eigenvalue of the Hessian at
$O$ is $-1$)
\be\label{sharp.2} 
\mathbb{E}_{\n_{B_-,B_+}}[\tau_{B_+}]=
\frac{2\pi e^{\frac{N}{4\epsilon}} \sqrt{\abs{\det(\nabla^2F_{\g,N}(O))}}}
      {\sqrt{\det(\nabla^2F_{\g,N}(I_-))}}
(1+O(\sqrt{\epsilon}\abs{\ln\epsilon}^3)). 
\ee
Finally, \eqref{eigenvalue.2} and \eqref{eigenvalue.3} give:
\be\label{det.1}
\det(\nabla^2F_{\g,N}(I_-))
=\prod_{k=0}^{N-1}\nu_{k,N}
=2\nu_{N/2,N}^{e(N)}\prod_{k=1}^{\lfloor\frac{N-1}{2}\rfloor}\nu_{k,N}^2
=2^{N}(1+\g)^{e(N)}\prod_{k=1}^{\lfloor\frac{N-1}{2}\rfloor}
\bigg(1+\frac{\g}{2\g^N_k}\bigg)^2
\ee

\be\label{det.2}
\abs{\det(\nabla^2F_{\g,N}(O))}
=\prod_{k=0}^{N-1}\l_{k,N}
=\l_{N/2,N}^{e(N)}\prod_{k=1}^{\lfloor\frac{N-1}{2}\rfloor}\l_{k,N}^2
=(2\g-1)^{e(N)}\prod_{k=1}^{\lfloor\frac{N-1}{2}\rfloor}
\bigg(1-\frac{\g}{\g^N_k}\bigg)^2. \ee Then, \be\label{cn.2} c_N
=\frac{\sqrt{\det(\nabla^2F_{\g,N}(I_-))}}
{\sqrt{\abs{\det(\nabla^2F_{\g,N}(O))}}}
=\bigg[1-\frac{3}{2+2\g}\bigg]^{\frac{e(N)}{2}}
\prod_{k=1}^{\lfloor\frac{N-1}{2}\rfloor}
\bigg[1-\frac{3}{2+\frac{\g}{\g^N_k}}\bigg]
\ee 
and Theorem
\ref{sharp} is proved.
\end{proof}

Let us point out that the use of these estimates is a major
obstacle to obtain a mean transition time starting from a single
stable point with uniform error terms. That is the reason why we
have introduced the equilibrium probability. However, there are
still several difficulties to be overcome if we want to pass to the
limit $N\uparrow \infty$.

\begin{itemize}
\item[(i)] We must show that the prefactor $c_N$ has a limit as
$N\uparrow\infty$.
\item[(ii)] The exponential term in the hitting
time tends to infinity with $N$. This suggests that one needs to
rescale the potential $F_{\g,N}$ by a factor $1/N$, or equivalently,
to increase the noise strength by a factor $N$.
\item[(iii)] One
will need uniform control of error estimates in $N$ to be able to
infer the metastable behavior of the infinite dimensional system.
This will be the most subtle of the problems involved.
\end{itemize}

\section{Large $N$ limit}

As mentioned above, in order to prove a limiting result as $N$
tends to infinity, we need to rescale the potential to eliminate
the $N$-dependence in the exponential. Thus henceforth we replace
$F_{\g,N}(x)$ by
\be \Eq(rescaled.1) G_{\g,N}(x)= N^{-1}F_{\g,N}(x). \ee

This choice actually has a very nice side effect. Namely, as we
always want to be in the regime where $\g\sim \g_1^N\sim N^{2}$,
it is natural to parametrize the coupling constant with a fixed
$\m>1$ as
\be\label{eigenvalue.4}
\g^N= \m \g^N_1
=\frac{\m}{2\sin^2(\frac{\pi}{N})} =\frac{\m
N^2}{2\pi^2}(1+o(1)).
\ee
Then, if we replace
the lattice by a lattice of spacing $1/N$, i.e. $(x_i)_{i\in\L}$ is
the discretization of a real function $x$ on $[0,1]$
($x_i=x(i/N)$), the resulting potential converges formally to 
\be
\Eq(rescaled.2) 
G_{\g^N,N}(x)\underset{N\to\infty}{\rightarrow}
\int_0^1 \left(\frac 14[x(s)]^4-\frac12[x(s)]^2 \right) ds +
\frac{\m}{4\pi^2}\int_0^1 \frac{\left[ x'(s)\right]^2}2ds,
\ee with
$x(0)=x(1)$.

In the Euclidean norm, we have $\|I_\pm\|_2=\sqrt{N}$, which suggests   to
rescale the size of neighborhoods. We consider, for $\r>0$,
the neighborhoods $B^N_\pm=B_{\r\sqrt{N}}(I_\pm)$.
The volume $V(B^N_-)=V(B^N_+)$ goes to $0$ if and only if
$\r<1/2\pi e$, so given such a $\r$, the balls $B^N_\pm$ are not
as large as one might think. Let us also  observe that
\be
\frac1{\sqrt{N}}\|x\|_{2}\underset{N\to
\infty}{\longrightarrow}\|x\|_{L^2[0,1]}=\int_0^1|x(s)|^2ds.
\ee
Therefore, if $x\in
B^N_+$ for all $N$, we get in the limit, $ \|x-1\|_{L^2[0,1]}<\r.
$

\bigskip
The main result of this paper is the following uniform
version of Theorem \ref{sharp} with a rescaled potential
$G_{\g,N}$.

\begin{theorem}\label{main}
Let $\mu\in]1,\infty[$, then there exists a constant, $A$, such that for all
$N\geq 2$ and all $\e>0$, 
\be\label{main.1}
\frac1N\E_{\n_{B^N_-,B^N_+}}[\t_{B^N_+}]=2\pi c_Ne^{1/4\e}(1+R(\e,N)),
\ee
where $c_N$ is defined in Theorem \thv(sharp) and 
$|R(\e,N)|\leq A \sqrt{\e|\ln \e|^3}$.
In particular,
\be\label{main.2}
\lim_{\e\downarrow 0}\lim_{N\uparrow\infty}\frac1N e^{-1/4\e} \E_{\n_{B^N_-,B^N_+}}[\t_{B^N_+}]=2\pi V(\m)
\ee
where
\be\label{main.3}
V(\m)=\prod_{k=1}^{+\infty}\Big[\frac{\mu k^2-1}{\mu k^2+2}\Big]<\infty.
\ee
\end{theorem}

\begin{remark} 
The appearance of the factor $1/N$ may at first glance seem disturbing. 
It corresponds however to the appropriate time rescaling when scaling 
the spatial coordinates $i$ to $i/N$  in order to recover the pde limit.
\end {remark}

The proof of this theorem will be decomposed in two parts:
\begin{itemize}
\item convergence of the sequence $c_N$ (Proposition \ref{convergence});
\item uniform control of the denominator (Proposition \ref{capacity})
and the numerator (Proposition \ref{numerator}) of Formula \eqref{key.4}.
\end{itemize}

\bigskip

\paragraph{Convergence of the prefactor $c_N$.}
Our first step will be to control the behavior of $c_N$ as $N\uparrow\infty$.
We prove the following:

\begin{proposition}\label{convergence}
The sequence $c_N$ converges: for $\m>1$, we set $\g=\m\g_1^N$,
then 
\be\label{convergence.1}
\lim_{N\uparrow\infty}c_N=V(\m), 
\ee
with $V(\m)$ defined in \eqref{main.3}.
\end{proposition}

\begin{remark}
This proposition immediately leads to
\begin{corollary}\label{sharp-limit}
For $\m\in]1,\infty[$, we set $\g=\m\g_1^N$, then
\be\Eq(time-final.1)
\lim_{N\uparrow \infty} \lim_{\e\downarrow 0}
\frac{e^{-\frac{1}{4\e}}}N\mathbb{E}_{\nu_{B^N_-,B^N_+}}[\t_{B^N_+}]=2\pi V(\m).
\ee
\end{corollary}

\noindent Of course such a result is unsatisfactory, since it does
not tell us anything about a large system with specified fixed
noise strength. To be able to interchange the limits regarding
$\e$ and $N$, we need a uniform control on the error terms.

\end{remark}

\begin{proof}[Proof of the proposition]
The rescaling of the potential introduces a factor $\frac1N$ for
the eigenvalues, so that  \eqref{sharp.2} becomes
\bea\label{convergence.3}\nonumber
\mathbb{E}_{\n_{B^N_-,B^N_+}}[\t_{B^N_+}] &=& \frac{2\pi
e^{\frac{1}{4\epsilon}}{N^{-N/2+1}\sqrt{\abs{\det(\nabla^2F_{\g,N}(O))}}}}
 {N^{-N/2}\sqrt{\det(\nabla^2F_{\g,N}(I_-))}
}
(1+O(\sqrt{\epsilon}\abs{\ln\epsilon}^3))
\\
&=&
2\pi N c_N e^{\frac{1}{4\epsilon}}(1+O(\sqrt{\epsilon}\abs{\ln\epsilon}^3)).
\eea

Then, with $u_k^N=\frac{3}{2+\mu\frac{\g^N_1}{\g^N_k}}\ $,
\be\label{cn.3}
c_N=\bigg[1-\frac{3}{2+2\mu\g^N_1}\bigg]^{\frac{e(N)}{2}}
\prod_{k=1}^{\lfloor\frac{N-1}{2}\rfloor}\bigg[1-u_k^N\bigg].
\ee
To prove the convergence, let us consider the
$(\g^N_k)_{k=1}^{N-1}$. For all $k\geq1$, we have
\be\label{eigenvalue.5} 
\frac{\g^N_1}{\g^N_k}
=\frac{\sin^2(\frac{k\pi}{N})}{\sin^2(\frac{\pi}{N})}
=k^2+(1-k^2)\frac{\pi^2}{3N^2}+o\bigg(\frac{1}{N^2}\bigg). 
\ee
Hence, $u_k^N\underset{N\to+\infty}{\longrightarrow}v_k=\frac{3}{2+\mu
k^2}$. Thus, we want to show that 
\be\label{cn.4}
c_N\underset{N\to+\infty}{\longrightarrow}\prod_{k=1}^{+\infty}(1-v_k)=V(\m).
\ee 
Using that, for $0\leq t\leq\frac\pi2$, 
\be
\label{sin.1} 0<t^2(1-\frac{t^2}{3})\leq\sin^2(t)\leq t^2, 
\ee
we get the following estimates for $\frac{\g^N_1}{\g^N_k}$: set
$a=\left(1-\frac{\pi^2}{12}\right)$, for $1\leq k\leq N/2$,
\be\label{eigenvalue.6} 
ak^2 =\bigg(1-\frac{\pi^2}{12}\bigg)k^2
\leq k^2\bigg(1-\frac{k^2\pi^2}{3N^2}\bigg)
\leq\frac{\g^N_1}{\g^N_k}
=\frac{\sin^2(\frac{k\pi}{N})}{\sin^2(\frac{\pi}{N})}
\leq\frac{k^2}{1-\frac{\pi^2}{3N^2}}. 
\ee 
Then, for $N\geq 2$ and
for all $1\leq k\leq N/2$,
\be\label{eigenvalue.6b} 
-\frac{k^4\pi^2}{3N^2}
\leq\frac{\g^N_1}{\g^N_k}-k^2
\leq\frac{k^2\pi^2}{3N^2\left(1-\frac{\pi^2}{3N^2}\right)}
\leq\frac{k^2\pi^2}{N^2}. 
\ee 
Let us introduce
\be
\label{convergence.4}
V_m=\prod_{k=1}^{\lfloor\frac{m-1}{2}\rfloor}(1-v_k),\quad
U_{N,m}=\prod_{k=1}^{\lfloor\frac{m-1}{2}\rfloor}\Big(1-u_k^N\Big).
\ee 
Then 
\be\label{convergence.5}
\Bigabs{\ln \frac{U_{N,N}}{V_N}}
=\Bigabs{\ln\prod_{k=1}^{\lfloor\frac{N-1}{2}\rfloor}\frac{1-u_k^N}{1-v_k}}
\leq\sum_{k=1}^{\lfloor\frac{N-1}{2}\rfloor}\Bigabs{\ln
\frac{1-u_k^N}{1-v_k}}. 
\ee 
Using \eqref{eigenvalue.6} and
\eqref{eigenvalue.6b}, we obtain, for all $1\leq k\leq N/2$, 
\be
\left|\frac{v_k-u_k^N}{1-v_k}\right| =
\frac{3\m\left|\frac{\g_1^N}{\g_k^N}-k^2\right|} {\left(-1+\m
k^2\right)\left(2+\m\frac{\g_1^N}{\g_k^N}\right)} \leq \frac{\m
k^4\pi^2} {N^2\left(-1+\m k^2\right)\left(2+\m ak^2\right)} \leq
\frac{C}{N^2}
\ee 
with $C$ a constant independent of $k$.
Therefore, for $N>N_0$, 
\be\label{convergence.6} \Bigabs{\ln
\frac{1-u_k^N}{1-v_k}}
=\left|\ln\left(1+\frac{v_k-u_k^N}{1-v_k}\right)\right|
\leq\frac{C'}{N^2}. 
\ee 
Hence
\be\label{convergence.7}
\Bigabs{\ln \frac{U_{N,N}}{V_N}}
\leq
\frac{C'}{N}
\underset{N\to+\infty}{\longrightarrow}0.
\ee
As $\sum\abs{v_k}<+\infty$, we get $\lim_{N\to+\infty}V_N=V(\m)>0$,
and thus \eqref{cn.4} is proved.
\end{proof}

\bigskip

\section{Estimates on capacities}

To prove Theorem \ref{main}, we prove uniform
estimates of the denominator and numerator of
\eqref{key.3}, namely the capacity and the mass
of the equilibrium potential.

\medskip

\subsection{Uniform control in large dimensions for capacities}

A crucial step is the control of the capacity. This will be done
with the help of the Dirichlet principle \eqv(cap.1). We will
obtain the asymptotics by using a Laplace-like method. The
exponential factor in the integral \eqref{cap.3} is largely
predominant at the points where $h$ is likely to vary the most,
that is around the saddle point $O$. Therefore  we need some good
estimates of the potential near $O$.

\medskip

\subsubsection{Local Taylor approximation}

This subsection is devoted to the quadratic approximations of the
potential which are quite subtle. We will make a  change of basis
in the neighborhood of the saddle point $O$ that will
diagonalize the quadratic part.

Recall that the potential $G_{\g,N}$ is of the form
\be\Eq(potential.6) 
G_{\g,N}(x)= -\frac 1{2N} (x,[\Id-\D]x) +\frac
1{4N} \|x\|_4^4. 
\ee 
where the operator $\D$ is given by
$\D=\g\left[\Id-\frac 12(\S+\S^*)\right]$ and $(\S x)_j = x_{j+1}$. The linear
operator $(\Id-\D)=-\nabla^2F_{\g,N}(O)$ has eigenvalues $-\l_{k,N}$
and eigenvectors $v_{k,N}$ with components $v_{k,N}(j)=\o^{jk}$, with
$\o=e^{i2\pi/N}$.

\noindent Let us change coordinates  by setting 
\be\Eq(fourier.1)
\hat x_j = \sum_{k=0}^{N-1} \o^{-jk}x_k. 
\ee 
Then the inverse
transformation is given by 
\be\Eq(fourier.2) 
x_k=\frac 1N
\sum_{j=0}^{N-1} \o^{jk}\hat x_j=x_k(\hat x). 
\ee

\noindent Note that the map $x\rightarrow \hat x$ maps $\R^N$ to
the set 
\be\Eq(space.1) 
\wh \R^N=\left\{\hat x\in \C^N:\hat
x_k=\overline{\hat x_{N-k}}\right\}
\ee 
endowed with the standard inner product on $\C^N$.

\noindent  Notice  that, expressed in terms of the variables $\hat x$, the potential
\eqref{potential.1} takes the form 
\be\label{potential.7}
G_{\g,N}(x(\hat x))= \frac 1{2N^2} \sum_{k=0}^{N-1} \l_{k,N} |\hat
x_k|^2+\frac 1{4N} \|x(\hat x)\|_4^4. 
\ee

\noindent Our main concern will be the control of the
non-quadratic term in the new coordinates. To that end, we
introduce the following norms on  Fourier space:

\be\Eq(fourier.3)
\|\hat x\|_{p,\FF}
=\left(\frac 1N
\sum_{i=0}^{N-1} |\hat x|^p\right)^{1/p}
=\frac1{N^{1/p}}\|\hat x\|_p.
\ee
The factor $1/N$ is
the suitable choice to make the map $x\rightarrow \hat x$ a
bounded map between $L^p$ spaces. This implies that the following estimates hold
(see \cite{Simon-Reed}, Vol. 1,
Theorem IX.8):

\begin{lemma}\label{Hausdorff-Young}
With the norms defined above, we have
\begin{itemize}
\item[(i)] the Parseval identity,
\be\label{parseval.1}
\|x\|_2=\|\hat x\|_{2,\FF},
\ee
and
\item[(ii)] the Hausdorff-Young inequalities:
for $1\leq q\leq 2$ and $p^{-1}+q^{-1}=1$, there exists a finite,
$N$-independent constant  $C_q$ such that
\be\Eq(fourier.4)
\|x\|_p\leq C_q\|\hat x\|_{q,\FF}.
\ee
In particular
\be\Eq(fourier.7)
\|x\|_4\leq C_{4/3} \|\hat x\|_{4/3,\FF}.
\ee
\end{itemize}
\end{lemma}

\bigskip
\noindent Let us introduce the change of variables, defined by the
complex vector $\, z$, as 
\be\label{change.1} 
z= {\hat x \over N}.
\ee 
\noindent  Let us remark that $\ z_0= {1\over N}
\sum_{k=1}^{N-1} x_k \in \mathbb{R}$.  In the variable $z$, the
potential takes the form 
\be\Eq(potential.8) 
\wt G_{\g,N}(z)=G_{\g,N}\left(x(N z)\right) =\frac 1{2}
\sum_{k=0}^{N-1}\l_{k,N} |z_k|^2+\frac 1{4N} \|x(N z)\|^4_4. 
\ee
Moreover, by \eqref{parseval.1} and \eqref{change.1}
\be\label{parseval.2}
\|x(Nz)\|_{2}^2=\|Nz\|_{2,\FF}^2=\frac1N\|Nz\|_{2}^2.
\ee 
In the
new coordinates the minima are now given by 
\be\label{change.2}
I_\pm=\pm (1,0,\dots,0). 
\ee 
In addition,
$z(B^N_-)=z(B_{\r\sqrt{N}}(I_-))=B_{\r}(I_-)$ where the last ball
is in the new coordinates.

\medskip
\noindent  Lemma \thv(Hausdorff-Young) will allow us to prove the
following important
estimates.
For $\d>0$, we set 
\be\Eq(neighborhood.1) 
C_\d=\left\{z\in
\wh\R^N:\,|z_k|\leq \d\frac{r_{k,N}}{\sqrt{|\l_{k,N}|}},\, 0\leq
k\leq N-1 \right\},
\ee
where  $\l_{k,N}$ are the eigenvalues of
the Hessian at $O$ as given in \eqv(eigenvalue.2) and $r_{k,N}$ are 
constants that will be specified below. Using
\eqref{eigenvalue.6}, we have, for $3\leq k\leq N/2$, 
\be
\l_{k,N}\geq k^2\left(1-\frac{\pi^2}{12}\right)\m-1. 
\ee 
Thus
$(\l_{k,N})$ verifies $\l_{k,N}\geq ak^2$, for $1\leq k\leq N/2$,
with some $a$, independent of $N$.

\medskip \noindent
The sequence $(r_{k,N})$ is constructed as follows. Choose an
increasing sequence, $(\r_k)_{k\geq 1}$, and set
\be\label{neighborhood.2}
\begin{cases}
r_{0,N}&=1\\
r_{k,N}&=r_{N-k,N}=\r_k,\quad 1\leq k\leq \left\lfloor\frac N2\right\rfloor.
\end{cases}
\ee
Let, for $p\geq 1$,
\be
K_p=\left(\sum_{k\geq1}\frac{\r_k^p}{k^p}\right)^{1/p}.
\ee
Note that if $K_{p_0}$ is finite then, for all $p_1>p_0$, $K_{p_1}$ is finite.
With this notation we have the following key estimate.

\begin{lemma}\label{norm}
For all $p\geq 2$, there exist finite  constants $B_p$, such that, for
$z\in C_\d$, 
\be\label{norm.1} \|x(Nz)\|^p_p\leq \d^p N B_p 
\ee
if
$K_q$ is finite, with $\frac1p+\frac1q=1$.
\end{lemma}

\begin{proof}
The Hausdorff-Young inequality (Lemma \ref{Hausdorff-Young}) gives us:
\be\label{norm.2}
\|x(Nz)\|_p\leq C_q\|Nz\|_{q,\FF}.
\ee
Since $z\in C_\d$, we get
\be\label{norm.3}
\|Nz\|^q_{q,\FF}\leq\d^q N^{q-1}\sum_{k=0}^{N-1}\frac{r_{k,N}^q}{\l_k^{q/2}}.
\ee
Then
\be\label{norm.4}
\sum_{k=0}^{N-1}\frac{r_{k,N}^q}{\l_k^{q/2}}
=
\frac1{\l_0^{q/2}}+2\sum_{k=1}^{\lfloor N/2\rfloor}\frac{r_{k,N}^q}{\l_k^{q/2}}
\leq
\frac1{\l_0^{q/2}}+\frac2{a^{q/2}}\sum_{k=1}^{\lfloor N/2\rfloor}\frac{\r_k^q}{k^q}
\leq
\frac1{\l_0^{q/2}}+\frac2{a^{q/2}}K_q^q=D_q^q
\ee
which is finite if  $K_q$ is finite. Therefore,
\be\label{norm.5}
\|x(Nz)\|^p_p\leq \d^pN^{(q-1)\frac pq}C_q^pD_q^p,
\ee
which gives us the result since $(q-1)\frac pq=1$.
\end{proof}

\noindent We have all what we need to estimate the capacity.

\medskip

\subsubsection{Capacity Estimates}

Let us now prove our main theorem.

\begin{proposition}\TH(capacity)
There exists a constant $A$, such that, for all $\e<\e_0$ and for all $N$,
\be\label{cap.4}
\frac{\capa\left(B^N_+,B^N_-)\right)}{N^{N/2-1}}
= \e
\sqrt{2\pi\e}^{N-2} \frac{1}{\sqrt{|\det(\nabla F_{\g,N}(0))|}}
\left(1+ R(\e,N)\right),
\ee
where $|R(\e,N)|\leq A\sqrt{\e |\ln \e|^3}$.
\end{proposition}

The proof will be decomposed into two lemmata, one for the upper
bound and the other for the lower bound. The proofs are quite
different but follow the same idea. We have to estimate some
integrals. We isolate a neighborhood around the point $O$ of
interest. We get an approximation of the potential on this
neighborhood, we bound the remainder and we estimate the integral
on the suitable neighborhood.

\medskip \noindent
In what follows, constants independent of $N$  are denoted $A_i$.

\newpage

\paragraph{Upper bound.}
The first lemma we prove is the upper bound for Proposition \thv(capacity).
\begin{lemma}\label{upper}
There exists a constant $A_0$ such that for all $\e$ and for all $N$,
\be\label{upper.1}
\frac{\capa\left(B^N_+,B^N_-\right)}{N^{N/2-1}}
\leq \e
\sqrt{2\pi\e}^{N-2} \frac{1}{\sqrt{|\det(\nabla F_{\g,N}(0))|}}
\left(1+ A_0\e|\ln \e|^2\right).
\ee
\end{lemma}

\begin{proof}
This lemma is proved in \cite{bovier04} in the
finite dimension setting.
We use the
same strategy, but here we take care to
  control
the integrals appearing  uniformly in the dimension.

We will denote the quadratic approximation of $\wt G_{\g,N}$ by $F_0$, i.e.
\be\label{name.1}
F_0(z)=\sum_{k=0}^{N-1}\frac{{\l_{k,N}}|z_k|^2}{2}
=-\frac{z_0^2}{2}+\sum_{k=1}^{N-1}\frac{{\l_{k,N}}|z_k|^2}{2}.
\ee

On $C_\d$, we can control the non-quadratic part through
Lemma \ref{norm}.
\begin{lemma}
\TH(u-approx)
There exists a constant $A_1$ and $\d_0$, such that for all $N$, $\d<\d_0$
and all $z\in C_\d$,
\be\Eq(u-approx.1)
\Big|\wt G_{\g,N}(z)-F_0(z)\Big|\leq A_1\d^4.
\ee
\end{lemma}

\begin{proof}
Using \eqref{potential.8}, we see that
\be\label{potential.9}
\wt G_{\g,N}(z)-F_0(z)=\frac 1{4N} \|x(N z)\|^4_4.
\ee
We choose a sequence $(\r_k)_{k\geq1}$  such that $K_{4/3}$
is finite.

Thus, it follows from Lemma \ref{norm}, with $A_1=\frac 14 B_4$, that
\be\Eq(u-approx.7) 
\left|\wt G_{\g,N}(z)  -  \frac 1{2}
\sum_{k=0}^{N-1} \l_{k,N} |z_k|^2\right|\leq A_1\d^4, 
\ee as
desired.
\end{proof}

\noindent We  obtain the upper bound of Lemma \ref{upper} by
choosing a test function $h^+$. We change coordinates from $x$ to
$z$ as explained in \eqref{change.1}. A simple calculation shows
that 
\be\Eq(gradient.1)
\|\nabla h(x)\|_2^2=N^{-1} \|\nabla \tilde
h(z)\|_2^2, 
\ee 
where $\tilde h(z)=h(x(z))$ under our coordinate change.

For $\d$ sufficiently small, we can ensure that, for $z\not\in C_{\d}$ with $|z_0|\leq \d$,
\be
\wt G_{\g,N}(z)
\geq F_0(z)
=-\frac{z_0^2}2+
\frac12\sum_{k=1}^{N-1}\l_{k,N}|z_k|^2
\geq -\frac{\d^2}2+2\d^2
\geq\d^2.
\ee
Therefore, the strip 
\be\Eq(green.aa)
S_{\d}\equiv\{x|\,x=x(Nz),\,|z_0|<\d\}
\ee
 separates $\R^N$
into two disjoint sets, one containing $I_-$ and the other one
containing $I_+$, and for $x\in S_d\setminus C_\d$, $G_{\g,N}(x)\geq \d^2$.

The complement of $S_{\d}$ consists of two connected
components $\G_+,\G_-$  which contain $I_+$ and $I_-$,
respectively. We define 
\be\label{upper.3} \tilde h^+(z)=
\begin{cases}
1&\text{for}\,z\in\G_-\\
0&\text{for}\,z\in\G_+\\
f(z_0)&\text{for}\,z\in C_{\d}\\
\text{arbitrary}&\text{on}\,S_{\d}\setminus C_{\d}\,\text{
but}\,\norm{\nabla h^+}_2\leq \frac{c}{\d}.
\end{cases},
\ee 
where $f$ satisfies $f(\d)=0$ and $f(-\d)=1$ and will be
specified later.

Taking into account the change of coordinates,
the Dirichlet form \eqref{cap.3} evaluated
on $h^+$ provides  the upper bound
\bea\label{upper.5}
\Phi(h^+)
&=&
N^{N/2-1}\e\int_{z((B^N_-\cup B^N_+)^c)}
e^{-\wt G_{\g,N}(z)/\e}\norm{\nabla \tilde h^+(z)}_2^2dz
\\\nonumber
&\leq&
N^{N/2-1}\left[\e\int_{C_{\d}}e^{-\wt G_{\g,N}(z)/\e}
\left(f'(z_0)\right)^2dz+
\e\d^{-2}c^2\int_{S_{\d}\setminus C_\d}
e^{-\wt G_{\g,N}(z)/\e}dz\right].
\eea

The first term will give the dominant contribution. Let us focus on it first.
We replace $\wt G_{\g,N}$ by $F_0$, using the bound
\eqv(u-approx.1), and for suitably  chosen $\d$, we obtain
\bea\label{upper.6}\nonumber
\int_{C_{\d}}e^{-\wt G_{\g,N}(z)/\e} \left(f'(z_0)\right)^2dz
&\leq&
\left(1+2A_1\frac{\d^4}{\e}\right)\int_{C_{\d}}e^{-
F_0(z)/\e} \left(f'(z_0)\right)^2dz
\\\nonumber
&=&\left(1+2A_1\frac{\d^4}{\e}\right)
\int_{D_{\d}}e^{-\frac 1{2\e}\sum_{k=1}^{N-1}\l_{k,N}
|z_k|^2}dz_1\dots dz_{N-1}
\\&&\quad\times
\int_{-\d}^{\d}\left(f'(z_0)\right)^2e^{z_0^2/2\e}dz_0.
\eea
Here we have used that we can write $C_\d$ in the form $[-\d,\d]\times D_\d$.
As we
want to calculate an infimum, we choose a function $f$ which
minimizes the integral $\int_{-\d}^{\d}\left(f'(z_0)\right)^2e^{z_0^2/2\e}dz_0$.
A simple
computation leads to the choice
\be\label{upper.4} 
f(z_0)=\frac{\int_{z_0}^{\d} e^{-t^2/2\e}dt}
{\int_{-\d}^{\d}e^{-t^2/2\e}dt}.
\ee
Therefore
\be\Eq(tralala)
\int_{C_{\d}}e^{-\wt G_{\g,N}(z)/\e} \left(f'(z_0)\right)^2dz \leq
\frac{\int_{C_{\d}}e^{-\frac 1{2\e}\sum_{k=0}^{N-1}|\l_{k,N}|
|z_k|^2}dz} {\Big(\int_{-\d}^{\d} e^{-\frac 1{2\e} z_0^2
}dz_0\Big)^2}\left(1+2A_1\frac{\d^4}{\e}\right).
\ee
Choosing
$\d=\sqrt{K \e|\ln \e |}$, a simple calculation shows that there
exists $A_2$ such that
\be\label{upper.7}
\frac{\int_{C_{\d}}e^{-\frac 1{2\e}\sum_{k=0}^{N-1}|\l_{k,N}|
|z_k|^2}dz} {\Big(\int_{-\d}^{\d} e^{-\frac 12 z_0^2
/\e}dz_0\Big)^2} \leq \sqrt{2\pi\e}^{N-2}
\frac{1}{\sqrt{|\det(\nabla F_{\g,N}(0))|}}(1+A_2\e).
\ee

The second term in \eqref{upper.5} is bounded above in the
following lemma.

\begin{lemma}\label{concentration}
For $\d=\sqrt{K\e|\ln(\e)|}$ and $\r_k=4k^{\a}$, with $0<\a<1/4$,
there exists $A_3<\infty$, such that for all $ N$ and $0<\e<1$,
\be\Eq(concentration.1)
\int_{S_{\d}\setminus C_\d}e^{-\wt G_{\g,N}(z)/\e}dz\leq
\frac{A_3 \sqrt{2\pi \e}^{N-2}}
{\sqrt{\abs{\det\left(\nabla^2F_{\g, N}(O)\right)}}}\e^{3K/2+1}.
\ee
\end{lemma}

\begin{proof}
Clearly, by \eqv(potential.8), 
\be\label{concentration.3}
\wt G_{\g,N}(z)\geq -\frac{z_0^2}{2}+\frac 1{2} \sum_{k=1}^{N-1}
\l_{k,N} |z_k|^2.
\ee
Thus
\bea\label{concentration.5}
\nonumber
&&\int_{S_{\d}\setminus C_\d}e^{-\wt G_{\g,N}(z)/\e}dz \\\nonumber
&\leq&
\int_{-\d}^{\d} dz_0 \int_{\exists_{k=1}^{N-1}: |z_k|\geq \d r_{k,N}/
\sqrt{\l_{k,N}}}dz_1\dots dz_{N-1}
e^{-\frac 1{2\e}\sum_{k=0}^{N-1}\l_{k,N} |z_k|^2}
\\\nonumber
&\leq&
\int_{-\d}^{\d}e^{+z_0^2/2\e}dz_0
\sum_{k=1}^{N-1}\int_{|z_k|\geq  \d r_{k,N}/\sqrt{\l_{k,N}}}
e^{-\l_{k,N} |z_k|^2/2\e}dz_k
\\\nonumber&&\quad\quad\times
\prod_{1\leq i\neq k\leq N-1}
\int_{\R}e^{-\l_{i,N} |z_i|^2/2\e}dz_i
\\ 
&\leq& 2 e^{\d^2/2\e} 
\sqrt{\frac {2\e} {\pi}}  \sqrt{\prod_{i=1}^{N-1}2\pi\e \l_{i,N}^{-1}}
\sum_{k=1}^{N-1} r_{k,n}^{-1} e^{-\d^2r_{k,N}^2/2\e}.
\eea
Now,
\bea\label{concentration.8}\nonumber
\sum_{k=1}^{N-1}{r_{k,N}^{-1}}e^{-\d^2 r_{k,N}^2/2\e}
&=&
\sum_{k=1}^{\lfloor\frac N2\rfloor}{r_{k,N}^{-1}}e^{-\d^2 r_{k,N}^2/2\e}
+\sum_{k=\lfloor\frac N2\rfloor+1}^{N-1}{r_{N-k,N}^{-1}}e^{-\d^2 r_{N-k,N}^2/2\e}\\
&\leq&
2\sum_{k=1}^{\infty}\r_k^{-1}e^{-\d^2 \r_k^2/2\e}.
\eea
We choose $\r_k=4 k^\a$ with $0<\a<1/4$
to ensure that $K_{4/3}$ is finite. With our choice for $\d$, the sum in 
\eqv(concentration.8) is then given by 
\be\label{concentration.9}
\frac 14\sum_{n=1}^{\infty} n^{-\a} \e^{8K  n^{2\a}}
\leq \frac 14\e^{2K} \sum_{n=1}^{\infty}\e^{6K n^{2\a}}\leq C \e^{2K},
\ee
since the sum over $n$ is clearly convergent.
Putting all the parts together, we get that
\be\label{concentration.12}
\int_{S_{\d}\setminus C_\d}e^{-\wt G_{\g,N}(z)/\e}dz\leq
C\e^{3K/2+1} \sqrt{2 \pi \e}^{N-2}\frac{1}
{\sqrt{\abs{\det\left(\nabla^2F_{\g,N}(O)\right)}}} 
\ee
 and Lemma \ref{concentration} is proven.
\end{proof}

\bigskip
Finally, using \eqref{tralala}, \eqv(upper.5),  \eqref{upper.7}, and
\eqref{concentration.1}, we obtain the  upper bound
\be\label{upper.8} 
\frac{\Phi(h^+)}{N^{N/2-1}} \leq
 \frac{\e\sqrt{2\pi\e}^{N-2}}{\sqrt{|\det(\nabla F_{\g,N}(0))|}}
(1+A_2\e)\left(1+2A_1\e|\ln\e|^2 + A_3'\e^{3K/2}\right) 
\ee
with the choice $\r_k=4k^\a$, $0<\a<1/4$ and
$\d=\sqrt{K\e|\ln\e|}$. Note that all constants are independent of $N$.
Thus Lemma \ref{upper} is proven.
\end{proof}

\bigskip

\paragraph{Lower Bound}
The idea here (as already used in \cite{BBI08}) is to get a
lower bound by restricting the state space to a narrow corridor
from $I_-$ to $I_+$ that contains the relevant paths and along
which the potential is well controlled.
We will prove the following lemma.

\begin{lemma}\label{lower}
There exists a constant $A_4<\infty$ such that for all $\e$ and for all $N$,
\be\label{lower.1}
\frac{\capa\left(B^N_+,B^N_-\right)}{N^{N/2-1}}
\geq \e
\sqrt{2\pi\e}^{N-2} \frac{1}{\sqrt{|\det(\nabla F_{\g,N}(0))|}}
\left(1- A_4 \sqrt{\e |\ln \e|^3}\right).
\ee
\end{lemma}

\begin{proof} Given a sequence, $(\r_k)_{k\geq 1}$, with $r_{k,N}$
defined as in \eqref{neighborhood.2}, we set
\be\label{lower.2}
\widehat{C}_{\d}
=
\left\{ z_0\in ]-1+\r,1-\r[,
|z_k|\leq \d\ r_{k,N}/\sqrt {\l_{k,N}}\right\}.
\ee
The restriction $|z_0|<1-\r$ is made to ensure
that $\widehat{C}_{\d}$ is disjoint from $B_\pm$
since in the new coordinates \eqref{change.1}
$I_\pm=\pm (1,0,\dots,0)$.

Clearly, if $h^*$ is the minimizer of the Dirichlet form, then
\be\label{lower.4}
\capa\left(B^N_-,B^N_+\right)=\Phi(h^*)
\geq\Phi_{\widehat{C}_{\d}}(h^*),
\ee
where $\Phi_{\widehat C_\d}$ is the Dirichlet
form for the process on $\widehat C_\d$,
\be\Eq(lower.5)
\Phi_{\widehat C_\d}(h)=\e\int_{\widehat{C}_{\d}}
e^{-G_{\g,N}(x)/\e}\|\nabla h(x)\|_2^2dx=N^{N/2-1}\e\int_{z(\widehat{C}_{\d})}
e^{-\wt G_{\g,N}(z)/\e}\|\nabla \tilde h(z)\|_2^2dz.
\ee
To get our lower bound we now use simply that 
\be
\|\nabla \tilde h(z)\|_2^2
=\sum_{k=0}^{N-1}\left|\frac{\partial \tilde h^*}{\partial z_k}\right|^2
\geq \left|\frac{\partial \tilde h^*}{\partial z_0}\right|^2,
\ee
so that
\be\label{lower.6}
\frac{\Phi(h^*)}{N^{N/2-1}}
\geq\e\int_{z(\widehat{C}_{\d})}
e^{-\wt G_{\g,N}(z)/\e}\Big|\frac{\partial \tilde h^*}
{\partial z_0}(z)\Big|^2dz
=\wt\Phi_{\widehat C_\d}(\tilde h^*)
\geq \min_{ h\in\HH}\wt\Phi_{\widehat C_\d}(\tilde h).
\ee
The remaining variational problem involves only functions depending on 
the single coordinate $z_0$, with the other coordinates, 
$z_\bot=(z_i)_{1\leq i\leq N-1}$, appearing only as parameters. 
The corresponding minimizer is readily found explicitly as
\be\label{lower.7}
\tilde h^-(z_0,z_\bot)=\frac{\int_{z_0}^{1-\r}e^{\wt G_{\g,N}(s,z_{\bot})/\e}ds}
{\int_{-1+\r}^{1-\r}e^{\wt G_{\g,N}(s,z_{\bot})/\e}ds}
\ee
and hence the capacity is bounded from below by
\be\label{lower.8}
\frac{\capa\left(B^N_-,B^N_+\right)}{N^{N/2-1}}
\geq \wt\Phi_{\wh C_\d}(\tilde h^-)
=\e\int_{\widehat{C}^{\bot}_{\d}}
\Big(\int_{-1+\r}^{1-\r}e^{\wt G_{\g,N}(z_0,z_{\bot})/\e}dz_0\Big)^{-1}dz_{\bot}.
\ee
Next, we have to evaluate the integrals in the r.h.s.  above.
The next lemma provides a suitable 
approximation of the potential on $\wh C_\d$.
Note that since  $z_0$ is no longer small, we only expand  in the
coordinates $z_\bot$.

\begin{lemma}\label{l-approx}
Let  $r_{k,N}$ be chosen as before with $\r_k=4k^\a$, $0<\a<1/4$. Then 
there exists a constant, $A_5$, and $\d_0>0$, such that, 
for all $N$ and $\d<\d_0$,
on $\wh C_\d$,
\be
\Eq(l-approx.1)
\left|\wt G_{\g,N}(z)-\left( -\frac 12z_0^2 + \frac 14 z_0^4
+\frac 1{2} \sum_{k=1}^{N-1}\l_{k,N} |z_k|^2
+z_0^2 f(z_\bot)\right)\right|
\leq A_5\d^3,
\ee
where 
\be\Eq(1-approx.1.1)
f(z_\bot)\equiv \frac 32\sum_{k=1}^{N-1}|z_k|^2.
\ee
\end{lemma}

\begin{proof}
We analyze  the non-quadratic part of the potential on $\wh C_\d$,
using \eqref{potential.8} and \eqref{fourier.2}
\be\label{l-approx.3} \frac{1}{N}\|x(Nz)\|_4^4
=\frac{1}{N}\sum_{i=0}^{N-1}|x_i(Nz)|^4
=\frac{1}{N}\sum_{i=0}^{N-1}\left|z_0+\sum_{k=1}^{N-1}\o^{ik}z_k\right|^4
=\frac{z_0^4}{N}\ \sum_{i=0}^{N-1}|1+u_i|^4 
\ee 
where $u_i=\frac
1{z_0}\sum_{k=1}^{N-1}\o^{ik}z_k$. Note that
$\sum_{i=0}^{N-1}u_i=0$ and
$u=\frac1{z_0}x\left(N(0,z_\bot)\right)$  and that for $z\in \wh\R^N$, 
$u_i$ is real. 
Thus
\be\label{l-approx.4}
\sum_{i=0}^{N-1}|1+u_i|^4=N+\sum_{i=0}^{N-1}\left(6u_i^2+4u_i^3 +u_i^4\right) ,
\ee
we get that
\be\label{l-approx.5}
\bigg|\frac{1}{N}\|x(Nz)\|_4^4
-z_0^4\Big(1+\frac6N \sum_{i=0}^{N-1}u_i^2\Big)
\bigg|
\leq
\frac {z_0^4}N\left(4\|u\|_3^3+\|u\|_4^4\right).
\ee
A simple computation shows that
\be\label{l-approx.6}
\frac6N \sum_i u_i^2=\frac 6{z_0^2}\sum_{k\neq0}
|z_k|^2.
\ee
Thus as $|z_0|\leq1$, we see that
\be\label{l-approx.7}
\left|\frac{1}{N}\|x(Nz)\|_4^4-z_0^4-6z_0^2\sum_{k\neq0}|z_k|^2\right|
\leq
\frac 1N\big(4\|x(N(0,z_\bot))\|_3^3+\|x(N(0,z_\bot))\|_4^4\big).
\ee
Using again   Lemma \ref{norm}, we get 
\bea\label{l-approx.8}
\|x(N(0,z_\bot))\|_3^3
&\leq&B_3N\d^3
\\\nonumber\|x(N(0,z_\bot))\|_4^4
&\leq&B_4N\d^4.
\eea
Therefore, Lemma \ref{l-approx} is proved, with $A_5=4 B_3+B_4\d_0$.
\end{proof}

We use Lemma \ref{l-approx} to obtain the upper bound
\bea\label{lower.9}
\int_{-1+\r}^{1-\r}e^{\wt G_{\g,N}(z_0,z_{\bot})/\e}dz_0\leq
&\exp\left(\frac{1}{2\e}\sum_{k\neq0}\l_{k,N}|z_k|^2+\frac{A_5\d^3}{\e}\right)g(z_\bot),
\eea
where
\be\label{lower.10}
g(z_\bot)
=
\int_{-1+\r}^{1-\r}\exp\left(-\e^{-1}\left(\frac{1}{2}z_0^2
-\frac{1}{4}z_0^4-z_0^2f(z_\bot)\right)\right)dz_0.
\ee
This integral is readily estimate via Laplace's method as
\be\Eq(lower.11)
g(z_\bot)=\frac{\sqrt{2\pi\e}}{\sqrt{1-2f(z_\bot)}}
\left(1+O(\e)\right)
= \sqrt{2\pi\e} \left(1+O(\e)+O(\d^2)\right). 
\ee
Inserting this estimate into \eqv(lower.8), 
it remains to carry out the integrals over the vertical coordinates which 
yields 
\bea \Eq(lower.12)\nonumber
\wt\Phi_{\widehat C_\d}(\tilde h^-)&\geq& \e
\int_{\widehat C_\d^\bot}
\exp\left(-\frac{1}{2\e}\sum_{k=1}^{N-1}\l_{k,N}|z_k|^2-\frac{A_5\d^3}{\e}\right)
 \frac1{\sqrt{2\pi\e}} \left(1+O(\e)+O(\d^2)\right)dz_\bot\\\nonumber
&=&
 \sqrt{\frac  \e{2\pi}}
\int_{\widehat C_\d^\bot}
\exp\left(-\frac{1}{2\e}\sum_{k=1}^{N-1}\l_{k,N}|z_k|^2\right)dz_\bot
 \left(1+O(\e)+O(\d^2)+O(\d^3/\e)\right).\\
\eea
The integral is readily bounded by
\bea\Eq(lower.13)\nonumber
&&\int_{\widehat C_\d^\bot}
\exp\left(-\frac{1}{2\e}\sum_{k=1}^{N-1}\l_{k,N}|z_k|^2\right)dz_\bot
\geq \int_{\R^{N-1}}
\exp\left(-\frac{1}{2\e}\sum_{k=1}^{N-1}\l_{k,N}|z_k|^2\right)dz_\bot
\\\nonumber
&&-\sum_{k=1}^{N-1} 
\int_\R dz_1\dots\int_{|z_k|\geq \d r_{k,N}/\sqrt{\l_{k,N}}}\dots
\int_{\R}dz_{N-1}
\exp\left(-\frac{1}{2\e}\sum_{k=1}^{N-1}\l_{k,N}|z_k|^2\right)
\\\nonumber
&&\geq  \sqrt{2\pi \e}^{N-1}
\prod_{i=1}^{N-1} \sqrt {\l_{i,N}}^{-1}\left(1-
\sqrt{\frac {2\e}\pi}\d^{-1}\sum_{k=1}^{N-1}r_{k,N}^{-1}e^{-\d^2 r_{k,N}^2/2\e}
\right)\\
&&
=  \sqrt{2\pi \e}^{N-1} \frac 1{\sqrt{\abs{\det F_{\g,N}(O)}}}
\left(1+O(\e^K)\right),
\eea
when $\d=\sqrt{K\e\ln \e}$ and $O(\e^K)$ uniform in $N$.
Putting all estimates together, we arrive at the assertion of 
Lemma \thv(lower).

\end{proof}

\subsection{Uniform estimate of the mass of the equilibrium potential}

We will prove the following proposition.
\begin{proposition}\label{numerator}
There exists a constant $A_6$ such that, for all
$\e<\e_0$ and  all $N$,
\be\Eq(numerator.1)             
\frac{1}{N^{N/2}}\int_{{B^N_+}^c}h^*_{B^N_-,B^N_+}(x)e^{- G_{\g,N}(x)/\e}dx
=\frac{\sqrt{2\pi\e}^{N}\exp\left(\frac1{4\e}\right)}{\sqrt{\det(\nabla F_{\g,N}(I_-))}}\left(1+ R(N,\e)\right),
\ee
where 
$
|R(N,\e)|\leq A_6 \sqrt{\e|\ln \e|^3}
$.
\end{proposition}

\begin{proof}

The predominant
contribution to the integral comes from
the minimum $I_-$, since around $I_+$ the harmonic function 
$h^*_{B^N_-,B^N_+}(x)$ vanishes.

The proof will go in two steps. We define the tube in the $z_0$-direction, 
\be\Eq(tube.1)
\wt C_\d\equiv \left\{z: \forall_{k\geq 1} |z_k|\leq \d r_{k,N}/\sqrt{\l_{k,N}}
\right\},
\ee
and show that the mass of the complement of this tube is negligible. 
In a second step we show that within that tube, only the neighborhood of
$I_-$ gives a relevant and indeed the desired contribution. The reason for 
splitting our estimates up in this way is that we have to use different ways 
to control the non-quadratic terms.

\begin{lemma}\TH(tube.2)
Let $r_{k,N}$ be chosen as before and let $\d=\sqrt{K\e|\ln \e|}$. Then there
 exists a finite numerical constant, $A_7$, such that for all $N$,
\be\label{numerator.1bis}
\frac{1}{N^{N/2}}\int_{{\wt C_\d}^c}e^{- G_{\g,N}(x)/\e}dx
\leq A_7\sqrt{2\pi\e}^{N}
\frac{e^{\frac1{4\e}}}{\sqrt{\det(\nabla F_{\g,N}(O))}}
\e^{K}.
\ee
The same estimate holds for the integral over the complement of the  set 
\be\Eq(last.2)
D_\d\equiv\left\{x:|z_0-1|\leq \d \lor |z_0+1|\leq \d\right\}.
\ee
\end{lemma}

\begin{proof} 
Recall that $z_0=\frac 1N\sum_{i=0}^{N-1}x_i$.  Then we can write
\be\label{numerator.2}
G_{\g,N}(x)
= -\frac 12z_0^2+\frac 14z_0^4
  +\frac 1{2N}(x,\D x) -\frac 1{2N}\|x\|_2^2 +\frac 12 z_0^2
+\frac 1{4N} \|x\|^4_4 -\frac 14 z_0^4.
\ee
Notice first that by applying the Cauchy-Schwartz inequality, it follows that
\be\Eq(tube.5)
 z_0^4 =N^{-4} \left(\sum_{i=0}^{N-1}x_i\right)^4
\leq N^{-2} \left(\sum_{i=0}^{N-1} x_i^2\right)^2
\leq N^{-1}\sum_{i=0}^{N-1}x_i^4.
\ee
Moreover, $N^{-1}\|x\|_2^2= \|z\|_2^2$, so that expressed in the variables $z$,
\bea\Eq(tube.6)
G_{\g,N}(x)
&\geq& -\frac 12z_0^2+\frac 14z_0^4 +\frac 1{2N}(x,\D x)
-\frac 12\sum_{k=1}^{N-1} z_k^2         \\\nonumber
&=& -\frac 12z_0^2+\frac 14z_0^4
+\frac 12\sum_{k=1}^{N-1} \l_{k,N}z_k^2.
\eea
Therefore, as in the estimate \eqv(concentration.5),
\bea\Eq(tube.7)
\nonumber
\int_{\wt C_\d^c}e^{-\wt G_{\g,N}(z)/\e}dz
&\leq&
\int_{\R}e^{-\e^{-1}(z_0^4/4-z_0^2/2)}dz_0
\sum_{k=1}^{N-1}\int_{|z_k|\geq  \d r_{k,N}/\sqrt{\l_{k,N}}}
e^{-\l_{k,N} |z_k|^2/2\e}dz_k
\\&&\quad\quad\times
\prod_{1\leq i\neq k\leq N-1}
\int_{\R}e^{-\l_{i,N} |z_i|^2/2\e}dz_i
\\\nonumber
&\leq& \int_{\R}e^{-\e^{-1}(z_0^4/4-z_0^2/2)}dz_0    \d^{-1}  \sqrt{\e}
 \sqrt{\prod_{i=1}^{N-1}2\pi\e \l_{i,N}^{-1}}
\sum_{k=1}^{N-1} r_{k,N}^{-1} e^{-\d^2r_{k,N}^2/2\e}
\\\nonumber
&\leq& \int_{\R}e^{-\e^{-1}(z_0^4/4-z_0^2/2)} \sqrt{\prod_{i=1}^{N-1}2\pi\e \l_{i,N}^{-1}}
C \e^K.
\eea
Since clearly, 
\be\Eq(tube.8)
 \int_{\R}e^{-\e^{-1}(z_0^4/4-z_0^2/2)}dz_0 = 2\sqrt{\pi \e} e^{1/4\e}(1+\OO(\e)),
\ee
this proves the first assertion of the  lemma. 

Quite clearly, the same bounds will show that the contribution from the set 
where $|z_0\pm 1|\geq \d$ are negligible, by just considering now the fact that
the range of the integral over $z_0$ is bounded away from the minima in the 
exponent. 
\end{proof}

Finally, we want to compute the remaining part of the integral in 
\eqv(numerator.1), i.e. the integral over
$\wt C_\d\cap \{x:|z_0+1|\leq \d\}$. 
Since the eigenvalues of the Hessian at $I_-$, $\nu_{k,N}$, 
are comparable to the eigenvalues  $\l_{k,N}$ for $k\geq 1$
in the sense that there is a
finite positive  constant, $c^2_\mu$, 
depending only on $\mu$, such that 
$\l_{k,N}\leq\nu_{k,N}\leq c^2_\mu\l_{k_N}$, and since $\nu_{0,N}=2$,
this set is contained in $C_{c_\mu\d}$, where
\be\label{numerator.4}
C_{\d}(I_-)\equiv \left\{z\in \wh \R^N: |z_0+1|\leq \frac{\d}{\sqrt {\nu_0}},\,
|z_k|\leq \d \frac{r_{k,N}}{\sqrt {\n_{k,N}}}\:1\leq k\leq N-1\right\}.
\ee

It is easy to verify that on 
$C_{\d}(I_-)$,
there exists a constant, $A_8$, s.t.
\be\label{numerator.5}
\|z-z(I_-)\|_2^2\leq\d^2\sum_{k=0}^{N-1}\frac{r_{k,N}^2}{\n_{k,N}}\leq\d^2A_8K_2^2.
\ee
and so,  for $\d=\sqrt{K\e|\ln \e|}$,
$C_{\d}(I_-)\subset z(B_-)$.

On $C_\d(I_-)$ we have the following quadratic approximation.
 
\begin{lemma}\label{n-approx}
For all $N$,
\be\label{n-approx.1}
 \wt G_{\g,N}(z)+\frac 14-\frac 1{2} \sum_{k=0}^{N-1}\n_{k,N} |z_k|^2
= R(z)
\ee
and there exists a constant $A_{9}$ and $\d_0$ such
that, for $\d<\d_0$, on $C_{\d}(I_-)$
\be\label{n-approx.2}
|R(z)|\leq A_{9}\d^3
\ee
where the constants  $(\r_k)$  are chosen as before such
that $K_{4/3}$ is finite.
\end{lemma}

\begin{proof} The proof goes in exactly the same way as in the previous cases 
and is left to the reader.
\end{proof}

With this estimate it is now obvious that 
\bea\Eq(last.1)
\int_{C_\d(I_-)}\tilde h^*_{B_-^N,B_+^N}(z)e^{-\wt G_{\g,N}(z)/\e}dz
&=&
\int_{C_\d(I_-)}e^{- \wt G_{\g,N}(z)/\e}dz
\\\nonumber 
&=&
e^{1/4\e} \frac{\sqrt{2\pi\e}^N}
{\sqrt{\det \nabla^2 F_{\g,N}(I_-)}}\left(1+ O(\d^3/\e)\right),
\eea
Using that $\tilde h^*_{B_-^N,B_+^N}(z)$ vanishes on $B_+^N$ and hence on $C_\d(I_+)$,
this estimate together with Lemma \thv(tube.2) proves the proposition.
\end{proof}

\subsection{Proof of Theorem \thv(main)}


\begin{proof}
The proof of Theorem \thv(main) is now an  obvious consequence of
\eqref{key.4} together with   Propositions \ref{capacity} and \ref{numerator}.
\end{proof}

\end{document}